\documentclass[12pt, reqno]{amsart}
\usepackage{ amsmath,amsthm, amscd, amsfonts, amssymb, graphicx, color}
\usepackage[bookmarksnumbered, colorlinks, plainpages]{hyperref}
\newtheorem{thm}{Theorem}[section]
\newtheorem{cor}[thm]{Corollary}
\newtheorem{lem}[thm]{Lemma}
\newtheorem{prop}[thm]{Proposition}

\newtheorem{exam}[thm]{Example}
\numberwithin{equation}{section}

\begin{document}

\title{On Zhou nil-clean rings}

\author{Marjan Sheibani Abdolyousefi}
\author{Nahid Ashrafi}
\author{Huanyin Chen}
\address{
Farzanegan Girls University of Semnan, Semnan, Iran}
\email{<sheibani@fgusem.ac.ir>}
\address{
Faculty of Mathematics\\Statistics and Computer Science\\  Semnan University\\ Semnan, Iran}
\email{<nashrafi@semnan.ac.ir>}
\address{
Department of Mathematics\\ Hangzhou Normal University\\ Hang -zhou, China}
\email{<huanyinchen@aliyun.com>}


\begin{abstract}
A ring $R$ is a Zhou nil-clean ring if every element in $R$ is the sum of two tripotents and a nilpotent that commute. In this paper, Zhou nil-clean rings are further discussed with an emphasis on their relations with polynomials, idempotents and 2-idempotents. We prove that a ring $R$ is a Zhou nil-clean ring if and only if for any $a\in R$, there exists $e\in {\Bbb Z}[a]$ such that $a-e\in R$ is nilpotent and $e^5=5e^3-4e$, if and only if for any $a\in R$, there exist idempotents $e,f,g,h\in {\Bbb Z}[a]$ and a nilpotent $w$ such that $a=e+f+g+h+w$, if and only if for any $a\in R$, there exist 2-idempotents $e,f\in {\Bbb Z}[a]$ and a nilpotent $w\in R$ such that $a=e+f+w$, if and only if for any $a\in R$, there exists a 2-idempotent $e\in {\Bbb Z}[a]$ and a nilpotent $w\in R$ such that $a^2=e+w$, if and only if $R$ is a Kosan exchange ring. As corollaries, new characterizations of strongly 2-nil-clean rings are thereby obtained.
\end{abstract}

\maketitle

\section{Introduction}
Throughout, all rings are associative with an identity. An element $a$ in a ring $R$ is (strongly) clean provided that it is the sum of an idempotent and a unit (that commute). A ring $R$ is (strongly) clean in case every element in $R$ is (strongly) clean. A ring $R$ is an exchange ring provided that for any $a\in R$, there
exists an idempotent $e\in R$ such that $e\in aR$ and $1-e\in (1-a)R$. Every (strongly) clean ring is an exchange ring, but the converse is not true (see~\cite[Proposition 1.8]{N}). As is well known, a ring $R$ is an exchange ring if and only if every idempotent lifts modulo every right ideal (see~\cite{CH}). An element $e$ in a ring is tripotent if $e^3=e$.
A ring $R$ is a Zhou nil-clean if every element in $R$ is the sum of two tripotents and a nilpotent that commute. Several elementary properties of such rings are established in ~\cite{Z}.

In this paper, we further discuss Zhou nil-clean rings with an emphasis on their relations with polynomials, idempotents and 2-idempotents. In Section 2 we establish new structure of Zhou nil-clean rings. We prove that a ring $R$ is a Zhou nil-clean ring if and only if for any $a\in R$, there exists $e\in {\Bbb Z}[a]$ such that $a-e\in R$ is nilpotent and $e^5=5e^3-4e$, if and only if for any $a\in R$, there exist idempotents $e,f,g,h\in {\Bbb Z}[a]$ and a nilpotent $w$ such that $a=e+f+g+h+w$.

An element $e\in R$ is a 2-idempotent if $e^2$ is an idempotent, i.e., $e^2=e^4$. Every tripotent in a ring is a 2-idempotent, but the converse is not true. For instance, every element in ${\Bbb Z}_4$ is a 2-idempotent, while $2\in {\Bbb Z}_4$ is not a tripotent. In Section 3, we characterize Zhou nil-clean rings by means of 2-idempotents. We prove that a ring $R$ is Zhou nil-clean if and only if for any $a\in R$, there exist 2-idempotents $e,f\in {\Bbb Z}[a]$ and a nilpotent $w\in R$ such that $a=e+f+w$, if and only if for any $a\in R$, there exists a 2-idempotent $e\in {\Bbb Z}[a]$ and a nilpotent $w\in R$ such that $a^2=e+w$. Let $R$ be a ring in which every element in $R$ is the sum of two commuting 2-idempotents. As an application, we prove that
every $n\times n$ matrix over $R$ is the sum of two tripotent matrices and a nilpotent.

An element $u$ in a ring is unipotent if $u-1$ is nilpotent. A ring $R$ is called a Kosan ring if for any $u\in U(R)$, $u^4$ is unipotent.
Finally, in Section 4, we prove that a ring $R$ is a Zhou nil-clean ring if and only if $R$ is an exchange Kosan ring.

We use $N(R)$ to denote the set of all nilpotents in $R$ and $J(R)$ the Jacobson radical of $R$. ${\Bbb N}$ stands for the set of all natural numbers.
${\Bbb Z}[u]=\{ f(u)~|~f(t)$ is a polynomial with integral coefficients $\}$.

\section{Structure Theorems}

The purpose of this section is to explore structure of of Zhou nil-clean rings. We begin with

\begin{lem} Let $R$ be a Zhou nil-clean ring with $5\in N(R)$. Then for any $a\in R$ there exists $e\in {\Bbb Z}[a]$ such that $a-e\in R$ is nilpotent and $e^5=5e^3-4e$.
\end{lem}
\begin{proof} $(1)\Rightarrow (2)$ Let $a\in R$. then $a-a^5\in N(R)$. Set $x=3a+a^2+a^4~\mbox{and}~y=3a-a^2-a^4$. Then
$x-x^3, y-y^3\in N(R)$.
Set $$b=\frac{x^2+x}{2},c=\frac{x^2-x}{2}; p=\frac{y^2+y}{2},q=\frac{y^2-y}{2}.$$ Then $b-b^2,c-c^2;p-p^2,q-q^2\in N(R)$.
Thus, we can find idempotents $e,f;g,h\in {\Bbb Z}[a]$ such that $$b-e,c-f;p-g,q-h\in N(R).$$ Since $x=b-c$ and $y=p-q$, we see that $x-(e-f); y-(g-h)\in N(R)$.
Therefore $$a=(x+y)-5a=(e-f)+(g-h)+w~\mbox{for some}~w\in N(R).$$

As $bc=\frac{x^4-x^2}{4}\in N(R)$, we see that $ef\in N(R)$; and so $ef=0$. As $pq\in N(R)$, likewise, $gh=0$.

Since $5\in R$ is nilpotent, we directly verify that
$$\begin{array}{c}
x^2+x\equiv -a+2a^2+a^3-2a^4, y^2+y\equiv 2a-a^3+a^4;\\
x^2-x\equiv -2a+a^3+a^4, y^2-y\equiv a+2a^2-a^3-2a^4 (\mbox{mod} N(R)).
\end{array}$$
Moreover, we have
$$\begin{array}{c}
bp\equiv 2a+a^2-2a^3-a^4, bq\equiv 0;\\
cp\equiv 0, cq\equiv -2a+a^2+2a^3-a^4 (\mbox{mod} N(R)).
\end{array}$$
Thus, we see that $eh, fg\in N(R)$; hence, $eh=fg=0$. Accordingly, we see that $ef=gh=eh=fg=0$, and then we check that
$$\begin{array}{c}
(e-f+g-h)^5=(e+f+g+h+2eg+2fh)^2(e-f+g-h)\\
=(e+f+g+h+14eg+14fh)(e-f+g-h)\\
=e-f+g-h+30(eg-fh).
\end{array}$$ Moreover, we have
$$\begin{array}{c}
(e-f+g-h)^3=(e+f+g+h+2eg+2fh)(e-f+g-h)\\
=e-f+g-h+6(eg-fh).
\end{array}$$ Let $\alpha=e-f+g-h$. Then $\alpha^5=\alpha+5(\alpha^3-\alpha)$. Hence, $\alpha^5=5\alpha^3-4\alpha$, as required.\end{proof}

We now ready to present a new characterization of a Zhou nil-clean ring.

\begin{thm} Let $R$ be a ring. Then the following are equivalent:
\end{thm}
\begin{enumerate}
\item [(1)] {\it $R$ is Zhou nil-clean.}
\vspace{-.5mm}
\item [(2)] {\it $a^5-5a^3+4a\in R$ is nilpotent for all $a\in R$.}
\vspace{-.5mm}
\item [(3)] {\it For any $a\in R$, there exists $e\in {\Bbb Z}[a]$ such that $a-e\in R$ is nilpotent and $e^5=5e^3-4e$.}
\end{enumerate}
\begin{proof} $(1)\Rightarrow (3)$ In view of~\cite[Theorem 2.11]{Z}, $R$ is isomorphic to $R_1,R_2,R_3$ or the product of these rings, where $R_1$ is strongly nil-clean and $2\in N(R_1)$; $R_2$ is strongly 2-nil-clean and $3\in N(R_2)$; $R_3$ is Zhou nil-clean and $5\in N(R_3)$.

Step 1. Let $a\in R_1$. In view of~\cite[Lemma 2.4]{Z}, there exists an idempotent $e\in {\Bbb Z}[a]$ such that $a-e\in N(R_1)$. We easily check that $e^5=5e^3-4e$.

Step 2. Let $a\in R_2$. By virtue of~\cite[Lemma 2.6]{Z}, there exists a tripotent $e\in {\Bbb Z}[a]$ such that $a-e\in N(R_1)$.
We easily check that $e^5=5e^3-4e$.

Step 3. Let $a\in R_1$. According to Lemma 2.1, there exists an idempotent $e\in {\Bbb Z}[a]$ such that $a-e\in N(R_1)$. We easily check that $e^5=5e^3-4e$.

Let $a\in R$. Combining the preceding steps, we can find $e\in {\Bbb Z}[a]$ such that $a-e\in R$ is nilpotent and $e^5=5e^3-4e$.

$(3)\Rightarrow (2)$ Let $a\in R$. Then there exists $e\in {\Bbb Z}[a]$ such that $w:=a-e\in R$ is nilpotent and $e^5=5e^3-4e$. Hence,
$(a-w)^5=5(a-w)^3-4(a-w)$. Thus, $a^5-5a^3+4a\in N(R)$, as required.

$(2)\Rightarrow (1)$ By hypothesis, $2^3\times 3\times 5=3^5-5\times 3^3+4\times 3\in N(R)$; hence, $2\times 3\times 5\in N(R)$. Write $2^n\times 3^n\times 5^n=0 (n\in {\Bbb N})$. Thus, $R\cong R_1\times R_2\times R_3$, where $R_1=R/2^nR, R_2=R/3^nR$ and $R_3=R/3^nR$.

Step 1. Let $x\in R_1$. As $2\in N(R_1)$, we see that $x^4(x-1)=x^5-x^4\in N(R_1)$, and so $x-x^2\in N(R_1)$. This shows that
$R_1$ is strongly nil-clean. Hence, $R_1$ is Zhou nil-clean, by~\cite[Proposition 2.5]{Z}.

Step 2. Let $x\in R_2$. As $3\in N(R_2)$, we see that $x(x^2-1)^2=x^5-2x^3+x\in N(R_2)$. This shows that $x-x^3\in N(R_2)$. Hence, $R_2$ is strongly 2-nil-clean. In view of~\cite[Proposition 2.5]{Z}, $R_2$ is Zhou nil-clean.

Step 3. Let $x\in R_3$. As $5\in N(R_3)$, we have $x-x^5\in N(R_3)$. In light of~\cite[Proposition 2.10]{Z}, $R_3$ is Zhou nil-clean.

Therefore $R$ is Zhou nil-clean.\end{proof}

\begin{cor} A ring $R$ has the identity $x=x^5$ if and only if
\end{cor}
\begin{enumerate}
\item [(1)] {\it $R$ is reduced;}
\vspace{-.5mm}
\item [(2)] {\it $R$ has the identity $x^5=5x^3-4x$.}
\end{enumerate}
\begin{proof} $\Longrightarrow$ Clearly, $R$ is reduced. In view of~\cite[Theorem 2.11]{Z}, $R$ is Zhou nil-clean. Let $x\in R$.
By virtue of Theorem 2.2, $x^5-5x^3+4x\in N(R)=0$, as required.

$\Longleftarrow$ In view of Theorem 2.2, $R$ is Zhou nil-clean. In light of~\cite[Proposition 2.8]{Z}, $J(R)$ is nil, and so $J(R)=0$. This completes the proof by ~\cite[Theorem 2.11]{Z}.\end{proof}

\begin{exam} Let $R={\Bbb Z}_{25}$. Then $R$ is Zhou nil-clean. But $2\in R$ can not be written as the sum of $f^5=f$ and a nilpotent.\end{exam}
\begin{proof} Clearly, $J(R)=5R$ is nil and $R/J(R)\cong {\Bbb Z}_5$. In light of~\cite[Theorem 2.11]{Z}, $R$ is Zhou nil-clean.
We check that $N(R)=\{0,5,10,15,20\}$, and then $(2-x)^5\neq 2-x$ for all $x\in N(R)$, and we are through.\end{proof}

We come now to characterize Zhou nil-clean rings by means of their idempotents. We

\begin{thm} Let $R$ be a ring. Then the following are equivalent:
\end{thm}
\begin{enumerate}
\item [(1)] {\it $R$ is Zhou nil-clean.}
\vspace{-.5mm}
\item [(2)] {\it For any $a\in R$, there exist idempotents $e,f,g,h\in {\Bbb Z}[a]$ and a nilpotent $w$ such that $a=e+f+g+h+w$.}
\vspace{-.5mm}
\item [(3)] {\it For any $a\in R$, there exist idempotents $e,f,g,h$ and a nilpotent $w$ that commute such that $a=e+f+g+h+w$.}
\end{enumerate}
\begin{proof} $\Longrightarrow$ Since $R$ is Zhou nil-clean, it follows by~\cite[Theorem 2.11]{Z} that $2^5-2\in N(R)$; hence, $30=2\times 3\times 5\in N(R)$.
Write $2^n\times 3^n\times 5^n=0$. Then $R\cong R_1\times R_2\times R_3$, where $R_1\cong R/2^nR, R_2\cong R/3^nR$ and $R_3=R/5^nR$.
Clearly, each $R_i$ is Zhou nil-clean. As $2\in N(R_1)$, it follows by~\cite[Proposition 2.5]{Z} that $R_1$ is strongly nil-clean, and so every element in $R_1$ is the sum of an idempotent and a nilpotent that commute. As $3\in N(R_2)$, it follows by~\cite[Proposition 2.8]{Z} that $R_2$ is strongly 2-nil-clean, and so every element in $R_2$ is the sum of two idempotents and a nilpotent that commute.

Let $c\in R_3$, and let $a=2-c$. Clearly, $R_3$ is a Zhou nil-clean ring with $5\in N(R_3)$. Construct
$e,f,g,h$ as in the proof of Lemma 2.1, we see that
$a=(e-f)+(g-h)+w$, where $e,f,g,h\in R$ are idempotents and $w\in R$ is a nilpotent that commute. Hence, $c=(1-e)+f+(1-g)+h-w$.
Thus, $c$ is the sum of four idempotents and a nilpotent that commute in $R_3$. Therefore every element in $R$ can be written in this form.

$\Longleftarrow$ Let $a\in R$. Then there exist idempotents $e,f,g,h$ and a nilpotent $w$ that commute such that $2-a=e+f+g+h+w$. Hence,
$a=(1-e)-f+(1-h)-g-w$. Set $s=(1-e)-f$ and $t=(1-h)-g$. We easily check that $s^3=s, t^3=t$ and $st=ts$. Thus, $a$ is the sum of two tripotents and a nilpotent that commute. Therefore $R$ is Zhou nil-clean.\end{proof}

\begin{cor} A ring $R$ has the identity $x=x^5$ if and only if
\end{cor}
\begin{enumerate}
\item [(1)] {\it $R$ is reduced;}
\vspace{-.5mm}
\item [(2)] {\it Every element is the sum of four commuting idempotents.}
\end{enumerate}
\begin{proof} $\Longrightarrow$ This is obvious by Corollary 2.3 and Theorem 2.5.

$\Longleftarrow$ In view of Theorem 2.5, $R$ is Zhou nil-clean. Thus, $R/J(R)$ has the identity $x^5=x$ and $J(R)$ is nil by~\cite[Theorem 2.11]{Z}.
Thus, $J(R)=0$, and then the result follows.\end{proof}

\vskip4mm We note that "four idempotents" in the proceeding corollary can not be replaced by "three idempotents". For instance, ${\Bbb Z}_5$.
We do have

\begin{thm} Let $R$ be a ring. Then the following are equivalent:
\end{thm}
\begin{enumerate}
\item [(1)] {\it $R$ is strongly 2-nil-clean.}\vspace{-.5mm}
\item [(2)] {\it For any $a\in R$, there exist idempotents $e,f$ and a nilpotent $w$ that commute such that $a=e+f+w$.}
\vspace{-.5mm}
\item [(3)] {\it For any $a\in R$, there exist idempotents $e,f,g$ and a nilpotent $w$ that commute such that $a=e+f+g+w$.}
\end{enumerate}
\begin{proof} $(1)\Rightarrow (2)$ This is proved in~\cite[Lemma 2.2]{CS}.

$(2)\Rightarrow (3)$ This is trivial.

$(3)\Rightarrow (1)$ Let $a\in R$. Then there exist idempotents $e,f,g$ and a nilpotent $w$ that commute such that $2-a=e+f+g+w$. Hence, $a=(1-e)-f+(1-g)-w$. Clearly, $((1-e)-f)^3=(1-e)-f$. Thus, $a$ is the sum of a tripotent, an idempotent and a nilpotent that commute. In light of~\cite[Theorem 2.3]{CS}, $R$ is strongly 2-nil-clean.\end{proof}

\section{Decomposition of 2-Idempotents and nilpotents}

The aim of this section is to characterize Zhou nil-clean rings by their 2-idempotents and nilpotents. The next lemma will enable us to take full advantage of decompositions of rings.

\begin{lem} Let $R$ be a ring in which every element is the sum of two 2-idempotents and a nilpotent that commute. Then $30\in R$ is nilpotent.\end{lem}
\begin{proof} Write $3=g+h+w$, where $g^2=g^4,h^2=h^4,w\in N(R)$ and $e,f,w$ commute with  another.
Let $e=g^3$ and $f=h^3$. Then $3=e+f+b$, where $e^3=e,f^3=f,b=w+(e-e^3)+(f-f^3)$. As $(e-e^3)^2=(f-f^3)^2=0$, we see that
$b\in N(R)$. Hence, $3-e=f+b$, and so $(3-e)^3=f^3+3bf(b+f)$. This implies that
$(3-e)^3-(3-e)\in N(R)$, i.e., $24-9e(3-e)=24-27e+9e^2\in N(R)$. We infers that
$$2^3\times 3\times (3-2e)=(3+e)(24-27e+9e^2)\in N(R).$$ Hence, we can find $w\in N(R)$ such that
$2^3\times 3^2=2^4\times 3e+w$, and so $2^9\times 3^6=2^{12}\times 3^3e^3+w'$ for $w'\in N(R)$. It follows that
$$2^9\times 3^4(2^2-3^2)\in N(R),$$ i.e., $2^9\times 3^4\times 5\in N(R)$. Therefore $30\in N(R)$, as asserted.\end{proof}

\begin{thm} Let $R$ be a ring. Then the following are equivalent:
\end{thm}
\begin{enumerate}
\item [(1)] {\it $R$ is Zhou nil-clean.}
\vspace{-.5mm}
\item [(2)] {\it For any $a\in R$, there exist 2-idempotents $e,f\in {\Bbb Z}[a]$ and a nilpotent $w\in R$ such that $a=e+f+w$.}
\vspace{-.5mm}
\item [(3)] {\it Every element in $R$ is the sum of two 2-idempotents and a nilpotent that commute.}
\end{enumerate}
\begin{proof} $(1)\Rightarrow (2)$ This is obvious by~\cite[Theorem 2.11]{CS}, as every tripotent in $R$ is a 2-idempotent.

$(2)\Rightarrow (3)$ This is trivial.

$(3)\Rightarrow (1)$ In view of Lemma 3.1, $30\in N(R)$. Write $2^n\times 3^n\times 5^n=0 (n\in {\Bbb N})$. Then $R\cong R_1\times R_2\times R_3$, where $R_1=R/2^nR, R_2=R/3^nR$ and $R_3=R/5^nR$. Set $S=R_2\times R_3$. Then $R\cong R_1\times S$.
 Step 1. Let $a^2=a^4$ in $R_1$. Then $a^2(1-a)(1+a)=0$. As $2\in N(R_1)$, we see that
$(a-a^2)^2\in N(R)$; hence, $a-a^2\in N(R_1)$. In light of~\cite[Proposition 2.5]{Z}, there exists an idempotent $e\in {\Bbb Z}[a]$ such that 
$a-e\in N(R_1)$. Thus, every element in $R_1$ is the sum of two idempotents and a nilpotent that commute. This shows that 
$R_1$ is strongly 2-nil-clean; hence, it is Zhou nil-clean.

Step 2. Let $a^2=a^4\in S$. Then $a(a-a^3)=0$, and so $(a-a^3)^2=a(a-a^3)(1-a^2)\in N(R)$. Hence, $a-a^3\in N(S)$. As $2\in U(S)$, it follows by~\cite[Lemma 2.6]{Z} that there exists $e^3=e\in S$ such that $a-e\in N(S)$.

Let $c\in S$. Then we can find 2-idempotents $b,c\in S$ such that $a-b-c\in N(S)$ and $a,b,c\in S$ commute.  By the preceding discussion, we have
tripotents $g\in {\Bbb Z}[b]$ and $h\in {\Bbb Z}[c]$ such that $b-g,c-h\in N(S)$. Therefore $a-g-h=(a-b-c)+(b-g)+(c-h)\in N(S)$ where $a,g,h$ commute.
Therefore every element in $S$ is the sum of two tripotents and a nilpotent in $S$. That is, $S$ is Zhou nil-clean.

Therefore $R$ is Zhou nil-clean, as asserted.\end{proof}

\begin{cor} Let $R$ be a ring. Then the following are equivalent:
\end{cor}
\begin{enumerate}
\item [(1)] {\it $R$ is strongly 2-nil-clean.}
\vspace{-.5mm}
\item [(2)] {\it Every element in $R$ is the sum of a 2-idempotent and a nilpotent that commute.}
\end{enumerate}
\begin{proof} $(1)\Rightarrow (2)$ This is obvious, by~\cite[Theorem 2.8]{CS}.

$(2)\Rightarrow (1)$ In view of Theorem 3.2, $R$ is Zhou nil-clean. Write $2=e+w$ where $e^3=e\in R, w\in N(R)$. Then
$2^3-2\in N(R)$, and so $6\in N(R)$. In light of~\cite[Lemma 3.5]{Y}, $R$ is strongly nil-clean.\end{proof}

\begin{thm} Let $R$ be a ring. Then the following are equivalent:
\end{thm}
\begin{enumerate}
\item [(1)] {\it $R$ is Zhou nil-clean.}\vspace{-.5mm}
\item [(2)] {\it For any $a\in R$, there exists a 2-idempotent $e\in {\Bbb Z}[a]$ and a nilpotent $w\in R$ such that $a^2=e+w$.}
\vspace{-.5mm}
\item [(3)] {\it For any $a\in R$, $a^2$ is the sum of a 2-idempotent and a nilpotent that commute.}
\end{enumerate}
\begin{proof} $(1)\Rightarrow (2)$ As in the proof of Theorem 3.2, $R\cong R_1\times S$ with $2\in N(R_1)$ and $2\in U(S)$.
Then $R$ is strongly nil-clean by~\cite[Proposition 2.5]{Z}. Let $a\in R$. Write $a=(a_1,a_2)$. Then $a_1^2$ is the sum of an idempotent and a nilpotent in $R_1$.
As $R$ is Zhou nil-clean, so is $R_2$. Hence, $a_2^5-a_2\in N(S)$. This implies that $(a_2^2)^3-a_2^2=a_2(a_2^5-a_2)\in N(S)$. As $2\in U(S)$,
it follows by~\cite[Lemma 2.6]{Z} that there exists $f^3=f\in {\Bbb Z}[a_2]$ such that $a_2^2-f\in N(S)$. Thus, we can find a tripotent $e\in {\Bbb Z}[a]$ such that $a^2-e\in N(R)$. Thus proving $(2)$ as every tripotent is a 2-idempotent.

$(2)\Rightarrow (3)$ This is obvious, as every tripotent is a 2-idempotent.

$(3)\Rightarrow (1)$ Let $a\in R$. Then there exists a 2-idempotent $e\in R$ such that $a-e\in N(R)$ and $ae=ea$. Hence,
$(a^2)^4-(a^2)^2\in N(R)$; hence, $a^8-a^4\in N(R)$. Thus, $(a^5-a)^4=a^3(a^5-a)(a^4-1)^3=(a^8-a^4)(a^4-1)\in N(R)$, i.e., $a^5-a\in N(R)$. In light of~\cite[Theorem 2.11]{Z}, $R$ is Zhou nil-clean.\end{proof}

\begin{cor} Let $R$ be a ring. Then the following are equivalent:
\end{cor}
\begin{enumerate}
\item [(1)] {\it $R$ is Zhou nil-clean.}\vspace{-.5mm}
\item [(2)] {\it For any $a\in R$, $a^4$ is the sum of an idempotent and a nilpotent that commute.}
\end{enumerate}
\begin{proof} $(1)\Rightarrow (2)$ Let $a\in R$. In view of~\cite[Theorem 2.11]{Z}, $a^5-a\in N(R)$; hence, $(a^4)^2-a^4=a^3(a^5-a)\in N(R)$. In light of~\cite[Lemma 3.5]{Y}, $a^4\in R$ is strongly nil-clean, as required.

$(2)\Rightarrow (1)$ By hypothesis, $a^4$ is strongly nil-clean, and so $(a^4)^2-a^4\in N(R)$. Hence, $a^3(a^5-a)\in N(R)$.
Thus, $(a^5-a)^4=(a^4-1)^3a^3(a^5-a)\in N(R)$.
We infer that $a^5-a\in N(R)$. Therefore $R$ is Zhou nil-clean, by~\cite[Theorem 2.11]{Z}.\end{proof}

Recall that a ring $R$ is of bounded index if there exists $n\in {\Bbb N}$ such that $x^n=0$ for all $x\in N(R)$.

\begin{lem} Let $R$ be a ring in which every element in $R$ is the sum of two commuting 2-idempotents. Then $R$ is of bounded index.\end{lem}
\begin{proof} In view of Theorem 3.4, $R$ is Zhou nil-clean. By using~\cite[Theoerm 2.11]{Z}, $R/J(R)$ has the identity $x^5=x$.
By using Jacobson Theorem, $R/J(R)$ is commutative. Hence, $N(R)\subseteq J(R)$.

Write $3=e+f$ where $e,f$ are commuting 2-idempotents. Set $r=e^3-e$ and $s=f^3-f$. Then $re=sf=0$.
We check that $$8(e+f)=24=3^3-3=(e+f)^3-(e+f)=3e^2f+3ef^2+r+s.$$ Multiplying by $ef$ from two sides, we get
$$8e^2f+8ef^2=3ef^2+3e^2f+3rf^2+3se^2.$$ It follows that $5(e^2f+ef^2)=3(se^2+rf^2).$
Thus, $$2^3\times 3\times 5=5\times (3e^2f+3ef^2+r+s)=9(se^2+rf^2)+5(r+s).$$ This implies that $2^3\times 3\times 5ef=0$.
Clearly, $$3^2=e^2+2ef+f^2~\mbox{and}~3^4=e^4+4e^3f+6e^2f^2+4ef^3+f^4.$$ We obtain
$$2^2\times 3\times 5\times (3^4-3^2)=(2e^2+3ef+2f^2-1)\times 2^3\times 3\times 5ef=0.$$ That is, $2^5\times 3^3\times 5=0$.
Write $R_1=R/2^5R,R_2=R/3^3R$ and $R_3=R/5R$. Then $R\cong R_1\times R_2\times R_3$.

Step 1. Let $b\in J(R_1)$. Write $b=e+f$ where $e,f\in R_1$ are 2-idempotents. Then $b^2-b^4=2ef-4e^3f-6e^2f^2-4ef^3$. As $2^5=0$ in $R_3$, we see that
$(b^2-b^4)^5=0$, and so $b^{10}=0$. Thus, $R_1$ is of bounded index $10$.

Step 2. Let $b\in J(R_2)$. Write $b=e+f$ where $e,f\in R_2$ are 2-idempotents. Then $b^3=e^3+f^3+3ef(e+f)$. Hence, $b-b^3=(e-e^3)+(f-f^3)-3ef(e+f)$.
It follows that $b^2(b-b^3)=f^2(e-e^3)+e^2(f-f^3)-3efb^3$ and $b^4(b-b^3)=f^4(e-e^3)+e^4(f-f^3)-3efb^5$.
Thus, $\big(b^4(b-b^3)-b^2(b-b^3)\big)^3=0$, as $3^3=0$ in $R_2$. This implies that $b^9=0$, and so $R_2$ is of bounded index $9$.

Step 3. Let $b\in J(R_3)$. Write $b=e+f$ where $e,f\in R_3$ are commuting 2-idempotents. Then
$b^5=e^5+5e^4f+10e^3f^2+5e^2f^3+10ef^4+f^5=e^5+f^5=e^3+f^3$ as $5=0$. Hence, $b-b^5=(e-e^3)+(f-f^3)$.
It follows that $b^2(b-b^5)=f^2(e-e^3)+e^2(f-f^3)$ and $b^4(b-b^5)=f^4(e-e^3)+e^4(f-f^3)$. As $e$ and $f$ are 2-idempotents, we see that
$b^2(b-b^5)=b^4(b-b^5)$; hence, $b^3=0$. That is, $R_3$ is of bounded index $3$.

Therefore $R$ is of bounded index.\end{proof}

With this information we can now express every matrix over such rings by means of tripotent and nilpotent matrices.

\begin{thm} Let $R$ be a ring in which every element in $R$ is the sum of two commuting 2-idempotents. Then every $n\times n$ matrix over $R$ is the sum of
two tripotent matrices and a nilpotent.\end{thm}
\begin{proof} In view of Theorem 3.2, $R$ is Zhou nil-clean. By virtue of Lemma 3.6, $R$ is of bounded index. Therefore we complete the proof, by ~\cite[Theorem 3.7]{MC}.\end{proof}

\section{Kosan rings}

This section is devoted to a collection of elementary properties of Kosan rings which will be used in the sequel. We start by

\begin{prop}
\end{prop}
\begin{enumerate}
\item [(1)] {\it Every finite direct product of Kosan rings is a Kosan ring.}
\vspace{-.5mm}
\item [(2)] {\it Every subring of a Kosan ring is a Kosan ring.}
\vspace{-.5mm}
\item [(3)] {\it If $R$ is a Kosan ring, then $eRe$ is a Kosan ring for all idempotents $e\in R$.}
\end{enumerate}
\begin{proof} $(1)$ Let $S=\prod\limits_{i=1}^{n}R_{i}$, where each $R_i$ is a Kosan ring, if $u=(u_1, u_2,\cdots, u_n)\in U(S)$, then each $u_i\in U(R_i)$ and by hypothesis $u_{i}^{4}=1+w_i$ for some nilpotent $w_i\in R_i$. So we have $u^4=(1,1,\cdots, 1)+(w_1, w_2, \cdots, w_n)$ is a unipotent.

$(2)$ Let $S$ be a subring of a Kosan ring $R$ and $u\in U(S)$, then $u\in U(R)$ and so $u^4=1_R+w$ for some $w\in N(R)$. We infer that $w=u^4-1_S\in N(S)$. Hence, $S$ is Kosan.

$(3)$ Let $u\in U(eRe)$ with $e=e^2\in R$. Then $u+1-e\in U(R)$. By hypothesis, $(u+1-e)^4=1+w$ for some $w\in N(R)$. Hence, $u^4=e+w$, and so $w=u^4-e\in eRe$. Hence, $w\in N(eRe)$, as desired.\end{proof}

We note that the infinite products of Kosan rings may be not a Kosan ring. For instance, let $R=\prod\limits_{n=1}^{\infty|}{\Bbb Z}_{5^n}$. Then each ${\Bbb Z}_{5^n}$ is Zhou nil-clean, but $R$ is not Zhou nil-clean. Clearly, $a=(2,2,\cdots ,2,\cdots )\in R$, while $a-a^5\in N(R)$, and we are done by~\cite[Theorem 2.11]{Z}.

\begin{lem} Let $I$ be a nil ideal of a ring $R$. Then $R$ is a Kosan ring if and only if so is $R/I$.\end{lem}
\begin{proof} $\Longrightarrow$ This is obvious as every unit lifts modulo every nil ideal.

$\Longleftarrow$ Let $u\in U(R)$, then $ \bar{u}^4=\bar{1}+\bar{w}$ for $\bar{w}\in N(R/I)$. Hence, $ u^4= 1+w+r$ for some $r\in I$. Here $w+r\in N(R)$.\end{proof}

We use $T_n(R)$ to denote the ring of all $n\times n$ upper triangular matrices over a ring $R$. We have

\begin{thm} Let $R$ be a ring. Then the following are equivalent:
\end{thm}
\begin{enumerate}
\item [(1)] {\it $R$ is a Kosan ring.}
\vspace{-.5mm}
\item [(2)] {\it $T_n(R)$ is a Kosan ring for all $n\in {\Bbb N}$.}
\vspace{-.5mm}
\item [(3)] {\it $T_n(R)$ is a Kosan ring for some $n\in {\Bbb N}$.}
\end{enumerate}
\begin{proof} $(1)\Rightarrow (2)$ Choose $I= \{
\left(
\begin{array}{cccc}
0&a_{12}&\cdots&a_{1n}\\
&0&\cdots &a_{2n}\\
&&\ddots&\vdots\\
&&&0
\end{array}
\right)\in T_n(R) | $ each $a_{ij}\in R\}$  Then $I$ is a nil ideal of $R$. As $T_n(R)/I\cong \prod\limits_{i=1}^{n}R_{i}$ be the direct product of rings
$R_i\cong R$, it follows by Proposition 4.1, that $\prod\limits_{i=1}^{n}R_{i}$ is a Kosan ring ring. In light of Proposition 4.2, $T_n(R)$ is a Kosan ring, as required.

$(2)\Rightarrow (3)$ This is trivial.

$(3)\Rightarrow (1)$ Let $e=diag(1,0,\cdots ,0)$. Then $R\cong eT_n(R)e$. We complete the proof by Proposition 3.1.
\end{proof}

\begin{exam} Let $R$ be a ring. Then $M_2(R)$ is not a Kosan ring.\end{exam}
\begin{proof} Assume that $M_2(R)$ is a Kosan ring. Since
$\left(
\begin{array}{cc}
1&1\\
-1&0
\end{array}
\right)=\left(
\begin{array}{cc}
0&-1\\
1&1
\end{array}
\right)^{-1}\in GL_2(R)$, $$A:=\left(
\begin{array}{cc}
-2&-1\\
1&-1
\end{array}
\right)=\left(
\begin{array}{cc}
1&1\\
-1&0
\end{array}
\right)^4-I_4\in N(M_2(R)).$$ Let $S=\{ m\cdot 1_R~|~m\in {\Bbb Z}\}$. Then $S$ is a commutative subring of $R$.
As $A\in N(M_2(S))$, we see that $det(A)=3\in N(S)$, and so $3\in N(R)$.

Since $2=3-1\in U(R)$, $\left(
\begin{array}{cc}
1&1\\
-1&1
\end{array}
\right)=\left(
\begin{array}{cc}
2^{-1}&-2^{-1}\\
2^{-1}&2^{-1}
\end{array}
\right)^{-1}\in GL_2(R)$, and then $$\left(
\begin{array}{cc}
-5&0\\
0&-5
\end{array}
\right)=\left(
\begin{array}{cc}
1&1\\
-1&1
\end{array}
\right)^4-I_4\in GL_2(R).$$ This implies that $5\in N(R)$, and so $1_R=(3\cdot 2-5)\cdot 1_R\in N(R)$, a contradiction.
This completes the proof.\end{proof}

Recall that a ring $R$ is local if $R$ has only one maximal right ideal, i.e., $R/J(R)$ is a division ring.

\begin{exam} A local ring $R$ is a Kosan ring if and only if
\end{exam}
\begin{enumerate}
\item [(1)] {\it $J(R)$ is nil;}
\vspace{-.5mm}
\item [(2)] {\it $R/J(R)\cong {\Bbb Z}_2, {\Bbb Z}_3$ or ${\Bbb Z}_5$.}
\end{enumerate}
\begin{proof} $\Longrightarrow$ As $(2,3)=1$, we see that $2\in U(R)$ or $3\in U(R)$. If $2\in U(R)$, then $2^4-1\in N(R)$, and so $3\times 5\in N(R)$.
If $3\in U(R)$, then $3^4-1\in N(R)$, and so $2\times 5\in N(R)$. As $(2,3,5)=1$, we see that $p\in N(R)$ where $p=2,3$ or $5$.
Let $x\in J(R)$. Then $1+x^p\equiv (1+x)^p\equiv 1 (mod~N(R))$, and so $x^p\equiv 0 (mod~N(R))$. That is, $x\in N(R)$.
Thus, $J(R)$ is nil.

Since $R$ is local, $R/J(R)$ is a Kosan division ring. Let $0\neq x\in R/J(R)$. Then $x^4=1$, and so $(x^2+1)(x-1)(x+1)=0$. Hence, $x=-1,1$ or $x^{2}=-1$.
Suppose that $y\not\in \{ 0,1,-1,2,-2\}$. Then $y^2=-1$, and so $y+1\neq 0,-1,1$. Hence, $(y+1)^2=-1$, i.e., $2y=-1$. Thus, $2y=y^2$, and then $y=2$,
a contradiction. Therefore $|R/J(R)|\leq 5$, and so $R/J(R)\cong {\Bbb Z}_2, {\Bbb Z}_3$ or ${\Bbb Z}_5$.

$\Longleftarrow$ Since ${\Bbb Z}_2, {\Bbb Z}_3$ and ${\Bbb Z}_5$ are all Kosan rings, we complete the proof by Lemma 4.2.\end{proof}

\begin{exam} Let $R={\Bbb Z_n} (n\geq 2)$. Then $R$ is a Kosan ring if and only if $n=2^k3^l5^{s} (k,l,s~\mbox{are nonnegitive integers} )$.\end{exam}
\begin{proof}$ \Longrightarrow$ Write $n=p_1^{r_1}\cdots p_s^{r_s}$ where $p_1,\cdots,p_s$ are primes and $r_1,\cdots,r_s\in {\Bbb N}$. Then $R\cong \bigoplus\limits_{i=1}^{s}{\Bbb Z}_{p_i^{r_i}}$.
As every homomorphic image of Kosan rings is a Kosan ring, each ${\Bbb Z}_{p_i^{r_i}}$ is a local Kosan ring. In light of Example 4.5,
${\Bbb Z}_{p_i^{r_i}}/J({\Bbb Z}_{p_i^{r_i}})\cong {\Bbb Z}_{p_i}$ is ${\Bbb Z}_2, {\Bbb Z}_3$ or ${\Bbb Z}_5$. Thus $p_i=2,3$ or $5$, as desired.

$\Longleftarrow$ As $n=2^k3^l5^m$, we see that ${\Bbb Z}_n\cong R_1\times R_2\times R_3$, where $R_1= {\Bbb Z}_{2^k}, R_2={\Bbb Z}_{3^l}$ and $R_3={\Bbb Z}_{5^m}$. It is obvious that each $J(R_i)$ is nil and $R_1/J(R_1)\cong {\Bbb Z}_2$, $R_2/J(R_2)\cong {\Bbb Z}_3$, and $R_3/J(R_3)\cong
{\Bbb Z}_5$. Then by Example 4.5, $R$ is a Kosan ring.\end{proof}

\section{Exchange Properties}

The class of exchange rings is
very large. It includes all regular rings, all $\pi$-regular rings,
all strongly $\pi$-regular rings, all semiperfect rings, all left or
right continuous rings, all clean rings, all unit $C^*$-algebras of
real rank zero and all right semi-artinian rings (see~\cite{CH}). We now characterize Zhou nil-clean rings in terms of their exchange properties.
We need an elementary lemma.

\begin{lem} Let $R$ be an exchange Kosan ring. Then $30\in R$ is nilpotent.\end{lem}
\begin{proof} Since $R$ is an exchange ring, there exists an idempotent $f\in 3R$  such that $1-f\in (1-3)R$. Write $f=3a$ for some $a\in R$. We may assume that $a=fa$, then $3a^2=3a$. Then $1-f=(1-3)b$ for some $b\in R$ with $b(1-f)=b$, therefore $(-2)b^2=b$. Now we have $3=3-(1-f)+(1-f)$. It is obvious that $3-(1-f)$ is a unit with inverse $a-b$.

Set $e=1-f$ and $u=3-(1-f)$. Then $3=e+u$. By hypothesis, $u^4=1+w$, so $(3-e)^4=1+w$, then $81-65e=1+w$ which implies that $80=65e+w$. Thus, $80e=65e+we$, hence, $15e=we$. Also $80=65e+w=60e+5e+w=5e+(4e+1)w$. Then $80^2=5^2e+10e(4e+1)w+(4e+1)^2w^2$, and $5\times 80=5^2e+5(4e+1)w$; whence, $80^2-5\times 80\in N(R)$. This implies $2^4\times 3\times 5^3\in N(R)$ and so $2\times 3\times 5\in N(R)$.\end{proof}

\begin{lem} Let $R$ be an exchange Kosan ring, then $J(R)$ is nil.\end{lem}
\begin{proof} By virtue of Lemma 5.1, $30\in R$. Write $30^n=0$ for some $n\in {\Bbb N}$. Then $R\cong R_1\times R_2\times R_3$  where $R_1=R/2^nR, R_2=R/3^nR$   and $R_3=R/5^nR$. Let $x\in J(R_1)$ so $1-x\in U(R_1)$. Let $u=1-x$, as $R_1$ is a Kosan ring, $u^4+1$ is in $N(R_1)$. We have $(1-u)^4=1+u^4+2(3u^2-2u-2u^3)\in N(R_1)$, as $2\in N(R_1)$, then $(1-u)^4\in N(R_1)$ which implies that $1-u\in N(R_1)$ and so $x\in N(R_1)$. Now let $x\in N(R_2)$ then, $(1+x)^4=1+w$ for some $w\in N(R_2)$, so $x^4+2(x^3+3x^2+2x)=x(x^3+2(2x^2+3x+2))$, since $3\in N(R_2), 2\in U(R_2)$, and we have $2x^2+3x+2\in U(R_2)$, so $x^3+2(2x^2+3x+2)\in U(R_2)$, which implies that $x\in N(R_2)$. For $R_3$, as $5\in N(R_3)$ we deduce that $2\in U(R_3)$, so in the similar way, we can prove that $J(R_3)$ is nil. Therefore $J(R)$ is nil.\end{proof}

\begin{lem} Let $R$ be an exchange Kosan ring. If $J(R)=0$, then $R$ is reduced.\end{lem}
\begin{proof} We claim that $N(R)=0$, if not there exists some $a\in R$
such that $a^2 = 0$. Since $R$ is an exchange ring and $J(R) = 0$,
by\cite[Theorem 2.1]{L}, $eRe\cong M_2(T)$ for some idempotent $e\in R$
and some ring $T$. As $R$ is a Kosan ring, then $eRe$ is also Kosan
which implies that $M_2(T)$ is a Kosan ring which is a contradiction with
Example 4.4. This shows that $N(R)=0$ and so R is reduced.\end{proof}

We have accumulated all the information necessary to prove the following.

\begin{thm} A ring $R$ is a Zhou nil-clean if and only if
\end{thm}
\begin{enumerate}
\item [(1)] {\it $R$ is an exchange ring;}
\vspace{-.5mm}
\item [(2)] {\it $R$ is a Kosan ring.}
\end{enumerate}
\begin{proof} $\Longrightarrow$ It is clear that every Zhou nil-clean ring is periodic and so it is strongly clean. Now let $u\in U(R)$, then $u^5-u\in N(R)$, $u(u^4-1)\in N(R)$, since $u$ is a unit then $u^4-1\in N(R)$. Thus, $R$ is a Kosan ring.

$\Longleftarrow$ By virtue of Lemma 5.2, $J(R)$ is nil. Set $S=R/J(R)$. Then $S$ is an exchange Kosan ring with $J(S)=0$.
In view of Lemma 5.3, $S$ is reduced. Thus, $S$ is isomorphic to a subdirect product of some domains $S_i$. We see that each $S_i$ is a homomorphic image of $S$; hence, $S_i$ is an exchange Kosan ring with trivial idempotents. In light of~\cite[Lemma 17.2.1]{CH} and Lemma 5.2, $S_i$ is local and $J(S_i)$ is nil. By virtue of Example 4.5, $S_i/J(S_i)\cong {\Bbb Z}_2, {\Bbb Z}_3$ or ${\Bbb Z}_5$. For any $x\in S_i$, $x-x^5\in J(S_i)\subseteq N(S_i)$. Let $\overline{a}\in S$. Then $a-a^5\in N(S)=0$, and so $R/J(R)$ has the identity $x^5=x$. Therefore $R$ is Zhou nil-clean, by~\cite[Theoerm 2.11]{Z}.\end{proof}

\begin{cor} A ring $R$ is Zhou nil-clean if and only if
\end{cor}
\begin{enumerate}
\item [(1)] {\it $R$ is clean;}
\vspace{-.5mm}
\item [(2)] {\it $R$ is a Kosan ring.}
\end{enumerate}
\begin{proof} $\Longrightarrow$ This is a direct result of Theorem 5.4.

$\Longleftarrow$ Since every clean ring is an exchange ring, we are through by Theorem 5.4.\end{proof}

\begin{cor} A ring $R$ is strongly 2-nil-clean if and only if
\end{cor}
\begin{enumerate}
\item [(1)] {\it $R$ is an exchange ring;}
\vspace{-.5mm}
\item [(2)] {\it The square of every unit in $R$ is a unipotent.}
\end{enumerate}
\begin{proof} $\Longrightarrow$ In light of~\cite[Proposition 1.8]{N}, $R$ is an exchange ring and $J(R)$ is nil.
Also by~\cite[Theorem 3.3]{CS}, $R/J(R)$ is
tripotent. Now let $u\in U(R)$, then $u-u^3\in J(R)\subseteq N(R)$. Hence,
$u^2-1\in N(R)$ which implies that $u^2$ is unipotent.

$\Longleftarrow$ Let $u\in U(R)$. Then $u^2\in U(R)$; hence, it is a unipotent. Thus, $R$ is a Kosan ring. In light of Theorem 5.4, $R$ is Zhou nil-clean.
It follows by Lemma 5.1, $30\in N(R)$, and so $5\times 6=0$.
Thus, $R\cong R_1\times R_2$, where $R_1=R/6R$ and $R_2=R/5R$. As $R_1$ is a Zhou nil-clean ring with $6\in N(R_1)$, it follows by~\cite[Theorem 2.12]{Z} that $R_1$ is strongly 2-nil-clean. On the other hand, $5\in N(R_2)$, and so $4=5-1\in U(R_2)$. Hence, $2\in U(R_2)$, and then $2^2-1\in N(R_2)$.
This implies that $6=2\times 3\in N(R_2)$, whence $1=6-5\in N(R_2)$, an absurd. We infer that $R_2=0$.
Therefore $R\cong R_1$ is strongly 2-nil-clean, as asserted.\end{proof}

\end{document}